\documentclass[11pt,onecolumn]{article}
\usepackage{etoolbox,booktabs}
\makeatletter
\@ifundefined{color@begingroup}%
  {\let\color@begingroup\relax
   \let\color@endgroup\relax}{}%
\def\fix@ieeecolor@hbox#1{%
  \hbox{\color@begingroup#1\color@endgroup}}
\usepackage{algorithm, algorithmic, setspace}
\usepackage{tikz}
 \usepackage{makecell} 
\usepackage{cite}
\usepackage{circuitikz}
\usepackage{graphicx} 
\usepackage{amsmath} 
\usepackage{amssymb}  
\usepackage{amsfonts}
\usepackage{color}
\usepackage{rotating}
\usepackage[english]{babel}
\usepackage{hyperref}
\usepackage{cleveref}
\usepackage{caption}
\usepackage{amsthm}

\newcommand{\argmin}[1]{\underset{#1}{\mathrm{argmin\,}}}

\newcommand{\xtrue}{\widetilde{x}}

\newcommand{\y}{\widetilde{y}}

\newcommand{\sign}{\text{sign}}
\newcommand{\xstar}{x^\star}
\newcommand{\lstar}{\lambda^\star}
\newcommand{\R}{\mathbb{R}}

\newcommand{\prox}{\mathrm{prox}}

\newtheorem{lemma}{Lemma}
\newtheorem{proposition}{Proposition}
\newtheorem{theorem}{Theorem}

\newtheorem{remark}{Remark}
\newtheorem{definition}{Definition}
\newtheorem{assumption}{Assumption}

\title{Feedback control of Lagrange multipliers\\ for nonsmooth constrained optimization}
\author{V. Cerone, S. M. Fosson, S. Pirrera, A. Re, D. Regruto }

\date{}

\begin{document}

\maketitle

\begin{abstract}
We develop a control-theoretic framework for constrained optimization problems with composite objective functions that include nondifferentiable terms.
Building on the proximal augmented Lagrangian formulation, we construct a plant whose equilibria correspond to the stationary points of the optimization problem.
By interpreting the Lagrange multipliers as control inputs, we design two feedback control laws that steer the resulting closed-loop system towards a stationary point of the optimization problem. This approach yields two novel optimization algorithms, for which we provide a theoretical analysis, establishing global exponential convergence under strong convexity assumptions. Finally, we demonstrate the effectiveness of the proposed methods through numerical experiments, benchmarking their performance against state-of-the-art approaches.
\end{abstract}
\section{Introduction}\label{sec:introduction}
Nonsmooth optimization enables the modeling of a wide range of problems in diverse engineering fields, including signal processing \cite{Combettes2011}, compressed sensing \cite{fou13,has15book}, deep learning \cite{mig24}, system identification \cite{fox20}, and control \cite{gal12,Nagahara23}.

Nonsmooth terms promote desirable structural properties in the solutions or enforce constraints. For example, the $\ell_p$ norms with $p\leq 1$ induce sparsity and can be used for feature selection; see  \cite{tib96, woo16}. Similarly, the nuclear norm and log-det heuristics encourage low-rank structure in matrices; see \cite{faz03}. In neural networks, nonsmooth semi-algebraic regularization is used to design quantized architectures; see \cite{bai19,mig24}. Finally, the indicator function naturally imposes set constraints on the optimization variables; see \cite{cent24}.

Nonsmooth optimization is challenging because nondifferentiable terms prevent using gradient-based algorithms. Proximal operator-based methods are widely adopted as an effective alternative; these include  the proximal gradient algorithms \cite{par13, Lions79} and the alternating direction method of multipliers (ADMM) \cite{boy10}. 

In many practical settings, nonsmooth problems involve  constraints that encode the physical laws governing the system. Applications range from constrained sparse optimization \cite{fou13} to switched system identification \cite{pao07} and robust state estimation \cite{paj14}. 
%
This work focuses on nonsmooth optimization with equality constraints, which is relatively underexplored in the literature.
A common workaround consists in embedding the constraints into the cost function \cite{par13} and apply proximal-gradient methods. However, this is demanding when the proximal operator lacks a closed-form solution or when the projection to the feasible set is computationally expensive. These difficulties motivate the investigation of alternative methodologies.

In this work, we address this challenge by adopting a feedback control theory approach to design continuous-time (CT) dynamics that converge to the solution of equality-constrained nonsmooth composite  problems. 
The idea of analyzing optimization algorithms through the lens of CT dynamical systems dates back to the seminal works \cite{arr58, kose56}, which introduce a CT Lagrangian-based approach for smooth constrained optimization known as primal-dual gradient dynamics (PDGD). As proved in \cite{qu19}, PDGD is exponentially stable in the presence of strongly convex, smooth cost functions with linear equality constraints.
In the CT framework, the recent work \cite{SIC_TAC_25} introduces a novel control-theoretic approach, called controlled multipliers optimization (CMO), that addresses equality-constrained smooth optimization problems. Specifically, CMO associates a fictitious plant with the optimization problem and, by interpreting the Lagrange multipliers as control inputs, designs a controller that steers the plant state toward a stationary point of the problem. This perspective enables the systematic design of optimization algorithms using tools from control theory.
Related approaches that interpret Lagrange multipliers as feedback controllers are explored in \cite{SIC_CDC_24},  \cite{Zhang25}, and \cite{Allibhoy24} that address smooth constrained optimization problems via proportional-integral (PI) control, feedback linearization, and control barrier functions, respectively.

Equality-constrained nonsmooth composite problems are addressed in \cite{cent25} by exploiting the CMO framework. Specifically, the authors propose a PI-controlled proximal gradient dynamics (PI-PGD), that  incorporates proximal operators to handle nondifferentiability in the definition of the fictitious plant.

A structural challenge in establishing strong theoretical guarantees for the approach in \cite{cent25} is related to the use of a nondifferentiable Lagrangian. In contrast, the approaches in  \cite{dhingra19, ding19, Ozaslan22,Dhingra22} use an alternative formulation based on the proximal augmented Lagrangian, which is differentiable and thus simplifies the theoretical analysis. 

In this paper, we propose a novel CMO-based, proximal augmented Lagrangian-based approach for nonsmooth constrained optimization.
As in \cite{SIC_TAC_25} and \cite{cent25}, we employ a PI control law for the multipliers associated with equality constraints, but differently from \cite{cent25}, we introduce the dual variable linked to the nonsmooth term and take advantage of the flexibility to design suitable control laws to steer such a variable.
%

The contribution of this paper is threefold. First, we develop two first-order optimization algorithms using feedback control design techniques. The first algorithm is based on a static control law for the Lagrange multipliers associated with the nonsmooth term, extending proximal gradient dynamics to equality-constrained problems. The second algorithm uses a dynamic control law that generalizes the nonsmooth primal-dual gradient dynamics introduced in \cite{dhingra19}.
Second, we analyze the convergence of the proposed methods in the strongly convex setting. Third, we present numerical experiments demonstrating the algorithms' performance, including comparisons with state-of-the-art methods and cases where convergence is not theoretically guaranteed.

\subsubsection*{Organization} 
Section~\ref{sec:probl_statem} formulates the problem and reviews the theory of proximal operators. Section~\ref{sec:proposed_approach} introduces the proposed control-theoretic framework. Section~\ref{sec:static} develops the static control method and establishes convergence results for strongly convex problems with linear constraints. Section~\ref{sec:dynamic} presents the dynamic control approach and proves its convergence in the strongly convex case with linear constraints. Section~\ref{sec:num_es} provides numerical experiments that illustrate the effectiveness of the proposed methods in various applications. Finally, Section~\ref{sec:conc} concludes the paper.

\subsubsection*{Notation}
Given a function $f: \R^n \rightarrow \R$, $\nabla f: \R^n \rightarrow \R^n$ is the gradient of $f$, and given $h: \R^n \rightarrow \R^m$, $J_h : \R^n \rightarrow \R^{m,n}$ is the Jacobian of $h$. $\mathrm{dom}(f)$ denotes the domain of $f$. Given a block strutured matrix $A = \begin{bmatrix}
    A_{11} & A_{12} \\ A_{21} & A_{22}
\end{bmatrix}$, $A/A_{22} \doteq A_{11}-A_{12}A_{22}^{-1}A_{21}$ is the Shur complement of $A$ with respect to $A_{22}$.

\allowdisplaybreaks
\section{Problem statement and background}\label{sec:probl_statem}

We consider equality-constrained nonsmooth composite optimization problems of the kind 
\begin{equation}
    \begin{aligned}
        \min_{x \in \R^{n}} \quad & f(x) + g(x) \\
        \textrm{s.t.} \quad & h(x) = 0,
    \end{aligned}
    \label{eq: problem statement}
\end{equation}
where $f(x) : \R^{n} \rightarrow \R$ and $h(x) : \R^{n} \rightarrow \R^m$ are continuously differentiable functions, while $g(x) : \R^{n} \rightarrow \R$ is a non-differentiable convex function.

Typically, $f(x)$ represents a cost function to be minimized. The approach proposed in this work applies to both convex or nonconvex $f$, while we rigorously prove its convergence under convexity conditions. The term $g(x)$ serves as a regularization term that enforce structural characteristics.
For example, $g(x) = \|x\|_1$ promotes the sparsity of $x$; see, e.g. \cite{tib96}. The indicator function $g(x) = \iota_{\mathcal{C}}(x)$ of a convex set $\mathcal{C}$, defined by $\iota_{\mathcal{C}}(x) = 0$ if $x \in \mathcal{C}$ and $\iota_{\mathcal{C}}(x) = + \infty$ otherwise, guarantees that $x \in \mathcal{C}$. 

Finally, the equality constraints $h(x)=0$ enforce structural relations between the optimization variables. 

In this work, we look for first-order solutions to Problem~\eqref{eq: problem statement}, i.e., points $(\xstar,\lstar)$ that satisfy the following definition:
\begin{definition}[Stationary point,~\cite{lue16}]
    A stationary point $(x^\star, \lambda^\star)$ of Problem~\eqref{eq: problem statement} is a point that satisfies the first-order optimality conditions:
    \begin{subequations}\label{foc}
        \begin{align}
              & 0 \in -\nabla f(x^\star) - \partial g(x^\star) - J_h^\top (x^\star) \lambda^\star \label{foc1.1}\\
              & h(x^\star) = 0, \label{foc1.2}
        \end{align}     
    \end{subequations}
    where $\partial g(x^\star) \subseteq \mathbb{R}^n$ is the subdifferential of $g$ at $x^\star$, i.e.,
    \begin{equation*}
        \partial g(x) \doteq \{y \in \R^n: g(z) \ge g(x) + y^\top (z - x) \,\, \forall z \in \operatorname{dom}(g)\}.
    \end{equation*}
\end{definition}
In what follows, we review some essential background for the development of the proposed approach.

\subsection{Proximal operators and Moreau envelopes}
One of the main challenges in solving Problem \eqref{eq: problem statement} stems from the nondifferentiability of $g$. We provide a brief overview on the proximal operators, that can be used to address this point. For a more extensive discussion, we refer the reader to \cite{par13}.

Given a closed proper convex function $g(x): \R^{n} \rightarrow \R \cup \{+ \infty\}$, the proximal operator $\prox_{\mu g}(x) : \R^{n} \rightarrow \R^n$ of the scaled function $\mu \, g(x)$, where $ \mu > 0$, is defined as:
\begin{equation}
        \prox_{\mu g} (v) =  \argmin{x \in \R^{n}} \left( g(x) + \frac{1}{2 \mu} \|x-v\|_2^2 \right).
\end{equation}

The Moreau envelope of $g$ is defined as:
\begin{equation}\label{def:moreau_env}
     M_{\mu g}(v) = \inf_{x \in \R^n} \left (g(x) + \frac{1}{2 \mu}   \| x-v \|^2_2 \right)
\end{equation}
and can be interpreted as a smoothed version of $g(x)$. It has domain $\R^n$ and is continuously differentiable, even when $g(x)$ itself is not. The sets of minimizers of $g(x)$ and $M_{\mu g}(x)$ coincide, which implies that minimizing $g$ is equivalent to minimizing $M_{\mu g}$.
The  proximal operator and the Moreau envelope are related, since $\prox_{\mu g}$ returns the unique point that achieves the infimum  that defines $M_{\mu g}$:
\begin{equation}
     M_{\mu g}(v) = g(\prox_{\mu g}(v)) + \frac{1}{2 \mu}\| \prox_{\mu g}(v)-v \|^2_2.
\end{equation}
The gradient of the Moreau envelope is
\begin{equation}\label{eq: grad moreau envelope}
    \nabla M_{\mu g}(v) = \frac{1}{\mu}(v - \prox_{\mu g}(v)).
\end{equation}

For example, if $g(x)$ is the $\ell_1$ norm, the proximal operator is the soft thresholding, defined component-wise as $\prox_{\mu \ell_1}(v_i) := \sign (v_i) \max ( |v_i| - \mu,0)$. The associated Moreau envelope is the Huber function $M_{\mu \ell_1}(v_i) = \{\frac{1}{2 \mu} v_i^2, |v_i| \leq \mu;|v_i| - \frac{\mu}{2}, |v_i| \geq \mu \}$ and the corresponding gradient is the saturation function $\nabla M_{\mu \ell_1}(v_i) =  \sign(v_i) \, \min \left( \frac{|v_i|}{\mu}, 1 \right)$.

Instead, when $g(x) = \iota_{\mathcal{C}}(x)$ is the indicator function of a closed nonempty convex set $ \mathcal{C} $, the proximal operator of $g$ is the Euclidean projection onto $\mathcal{C}$, $\prox_{g}(v) = \Pi_{\mathcal{C}}(v) = \arg\min_{x \in \mathcal{C}} \|x - v\|_2$. The gradient of the corresponding Moreau envelope does not always have a closed form, but can be computed employing \eqref{eq: grad moreau envelope}.

\subsection{Controlled multipliers optimization}
The CMO framework proposed in \cite{SIC_TAC_25} develops a CT methododolgy for solving constrained optimization problems of the form
\begin{equation}\label{eq: problema cmo25}
    \min_{x \in \mathbb{R}^n} \quad f(x) \quad \text{s.t.} \quad h(x) = 0.   
\end{equation}
where both $ f : \mathbb{R}^n \to \mathbb{R} $ and $ h : \mathbb{R}^n \to \mathbb{R}^m $ are differentiable.
The Lagrangian associated with \eqref{eq: problema cmo25} is
\begin{equation}
    \mathcal{L}(x) = f(x) + \lambda^\top h(x)
\end{equation}
where $\lambda \in \mathbb{R}^m$ represents the vector of Lagrange multipliers.

The core idea of CMO is to associate Problem \eqref{eq: problema cmo25} with a dynamical system $ \mathcal{P} $ whose state $ x(t) \in \mathbb{R}^n $ corresponds to the optimization variable. The input of $ \mathcal{P} $ are the Lagrange multipliers $\lambda(t)$ while the system output are $y(t) = h(x(t))$:

\[
\mathcal{P} :
\begin{cases}
\dot{x}(t) = -\nabla f(x(t)) - J_h(x(t))^\top \lambda(t) \\
y(t) = h(x(t))
\end{cases}
\]

As illustrated in Fig.~\ref{fig: cmo general configuration}, this perspective allows for the design of a feedback controller $ \mathcal{K} $ that enforces convergence of the closed-loop trajectories, i.e.,

\begin{equation}
\lim_{t \to \infty} x(t) = x^*, \quad
\lim_{t \to \infty} y(t) = 0.
\end{equation}

Different choices of the control law characterizing $ \mathcal{K} $ result in   different continuous-time optimization algorithms. In particular, \cite{SIC_TAC_25} develops controllers based on feedback linearization and PI control to solve problem \eqref{eq: problema cmo25}.

In~\cite{cent25}, the CMO approach is applied to problems involving a composite nonsmooth objective function as  in Eq.~\eqref{eq: problem statement}, introducing the nonsmooth Lagrangian
\begin{equation}
    \mathcal{L}(x) = f(x) + g(x) + \lambda^\top h(x).
\end{equation}
The corresponding plant state equation is defined using the proximal operator of $g$. 

In this work, we propose an alternative differentiable approach based on Moreau envelopes.

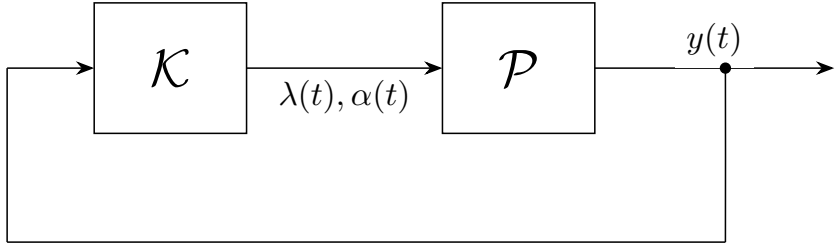
\begin{figure}
\centering
\resizebox{0.9\columnwidth}{!}{
\begin{circuitikz}[scale=0.8, transform shape] 
\tikzstyle{every node}=[font=\normalsize]
\draw  (3,9.75) rectangle  node {\LARGE $\mathcal{K}$} (4.75,8.25);
\draw  (3,9.25) rectangle (3,9.25);
\draw  (7,9.75) rectangle  node {\LARGE $\mathcal{P}$} (8.75,8.25);
\draw [->, >=Stealth] (4.75,9) -- (7,9)
node[pos=0.5,below, fill=white, align=center]
{$\lambda(t),\alpha(t)$};
\draw [->, >=Stealth] (8.75,9) -- (11.5,9)
node[pos=0.5,above, fill=white, align=center]
{$y(t)$};
\node at (10.25,9) [circ] {};
\draw (10.25,9) to[short] (10.25,7);
\draw (10.25,7) to[short] (2,7);
\draw [-, >=Stealth] (2,7) -- (2,9);
\draw [->, >=Stealth] (2,9) -- (3,9);
\end{circuitikz}
}%
\caption{Scheme of the CMO approach: the dynamical system $\mathcal{P}$ associated with the optimization problem is driven to equilibrium by the Lagrange multipliers that act as control inputs, while the constraints are the outputs.}
\label{fig: cmo general configuration}
\end{figure}
\section{Proposed approach}
\label{sec:proposed_approach}
In this section, we develop the proposed approach, building on the CMO framework \cite{SIC_TAC_25} and on the theory of Moreau envelopes and their properties.
\subsection{Proximal augmented Lagrangian}
A possible approach for handling a nonsmooth term in an optimization problem is to introduce an additional optimization variable that decouples the smooth and nonsmooth terms, followed by the definition of an augmented Lagrangian. The work \cite{dhingra19} develops this method for unconstrained problems. In the following, we extend it to equality-constrained problems.

By introducing the auxiliary variable $z \in \R^n$, Problem~\eqref{eq: problem statement} is equivalently formulated as:
\begin{equation}\label{eq: problem with additional constraint}
    \begin{aligned}
        \min_{x,z \in \R^{n}} \quad & f(x) + g(z) \\
        \textrm{s.t.} \quad & x - z = 0 \\
        & h(x) = 0
    \end{aligned}
\end{equation}
Given Problem \eqref{eq: problem with additional constraint}, inspired by  \cite[Theorem~1]{dhingra19}, we define the augmented Lagrangian 
%
\begin{equation}\label{eq: lagrangiana ns-cmo0}
\begin{aligned}
    \mathcal{L}_{\mu}(x, z, \alpha, \lambda) &\doteq  f(x) + g(z) + \alpha^\top(x-z) + \\
    &+\frac{1}{2\mu} \|x - z\|_2^2 + \lambda^\top h(x).
\end{aligned} \end{equation}

The variable $\alpha \in \R^n$ represents the vector of Lagrange multipliers associated with the constraint $x=z$, while $\lambda \in \R^m$ is the Lagrange multiplier associated with $h(x) = 0$.


Completion of squares on Eq.~\eqref{eq: lagrangiana ns-cmo0} gives
\begin{equation}\label{eq:aug_lag_sq_completed}
\begin{aligned}
\mathcal{L}_{\mu}(x, z, \alpha, \lambda) &= f(x) + g(z) + \frac{1}{2\mu} \|z - (x + \mu\alpha)\|_2^2 + \\ 
& \qquad - \frac{\mu}{2} \|\alpha\|_2^2 + \lambda^\top h(x).
\end{aligned}
\end{equation}
%

The term $\|z - (x + \mu\alpha)\|_2^2$ enables us to reformulate \eqref{eq:aug_lag_sq_completed} in terms of the Moreau envelope: minimization of $\mathcal{L}_{\mu}(x,z,\alpha,\lambda)$ with respect to $z$ yields
\begin{equation}\label{eq: z star}
   z^\star = \argmin{z\in\R^n}\mathcal{L}_{\mu}(x,z,\alpha,\lambda)  = \prox_{\mu g}(x + \mu \alpha).
\end{equation}
By replacing ~\eqref{eq: z star} in ~\eqref{eq:aug_lag_sq_completed} and using the Moreau envelope definition in Eq.~\eqref{def:moreau_env}, we derive the proximal augmented Lagrangian:
\begin{equation}\label{eq: lagrangiana ns-cmo}
    \mathcal{L}_{\mu} (x,\alpha, \lambda) = f(x) + M_{\mu g}(x + \mu \alpha) - \frac{\mu}{2} \| \alpha \|^2 + \lambda^\top h(x).
\end{equation}
%

The first order optimality conditions of Problem \eqref{eq: problem with additional constraint} are 
\begin{subequations}\label{foc_p2}
        \begin{align}
              & 0 \in -\nabla f(x^\star) - \partial g(z^\star) - J_h^\top (x^\star) \lambda^\star \label{foc2.1} \\ 
              & 0 = x^\star - z^\star \label{foc2.3} \\ 
              & 0 = h(x^\star).
        \end{align}     
    \end{subequations}

By substituting \eqref{foc2.3} into \eqref{foc2.1}, we recover the first-order optimality conditions of Problem \eqref{eq: problem statement}, which proves the equivalence between the two formulations.

In the following result, we illustrate the relationship between saddle points of \eqref{eq: lagrangiana ns-cmo} and stationary points of Problem \eqref{eq: problem statement}.

\begin{proposition}\label{proposition 1}
    Let $(x^\star, \alpha^\star, \lambda^\star)$ be a saddle point of $\mathcal{L}_{\mu} (x,\alpha, \lambda)$, i.e., 
    \begin{subequations}\label{eq: saddle point}
    \begin{align}
      \nabla f(x^\star) + \nabla M_{\mu g}(x^\star + \mu \alpha^\star) + J_h^\top (x^\star)\lambda^\star &= 0 \label{eq: saddle 1} \\[3pt] 
       \mu \nabla M_{\mu g}(x^\star + \mu \alpha^\star) - \mu \alpha^\star &= 0 \label{eq: saddle 2} \\[3pt]
      h(x^\star) &= 0. \label{eq: saddle 3}
    \end{align}
    \end{subequations}
    Then, $(x^\star, \lambda^\star)$ is a stationary point of Problem \eqref{eq: problem statement}.
\end{proposition}

\begin{proof}
    Exploiting the definition of $\nabla M_{\mu g}$ in \eqref{eq: grad moreau envelope}, we observe that \eqref{eq: saddle 2} implies:
    \begin{equation}
        x^\star = \prox_{\mu g}(x^\star + \mu \alpha^\star).\label{eq:x_star_eq_prox_lemm1}
    \end{equation}
    Using the subdifferential characterization of the minimum of a convex function \cite{rockafellar1970}, that states that  $\tilde{x} = \prox_{g}(v)$ if and only if $0 \in \partial g(\tilde{x}) + \tilde{x} - v$, we can rewrite~\eqref{eq:x_star_eq_prox_lemm1} as
    \begin{equation}\label{eq: dg}
        0 \in \partial g(x^\star) - \alpha^\star.
    \end{equation}
    Substitution of \eqref{eq: saddle 2} in \eqref{eq: saddle 1} leads to    \begin{equation}\label{eq:cond_proposition1}
        \nabla f(x^\star) +  \alpha^\star + J_h^\top (x^\star)\lambda^\star = 0.
    \end{equation}
    Finally, using the inclusion in \eqref{eq: dg} to replace $ \alpha^\star$ in \eqref{eq:cond_proposition1}, we recover \eqref{foc1.1}.
    %
    This condition, along with \eqref{eq: saddle 3}, corresponds to the optimality conditions associated with problem \eqref{eq: problem statement}, as stated in Eq.~\eqref{foc}. 
\end{proof}

The proximal augmented Lagrangian defined in Eq.~\eqref{eq: lagrangiana ns-cmo} extends the proximal augmented Lagrangian proposed in \cite{dhingra19}, i.e., 
\begin{equation}\label{eq: prox_aug_noeq}
    \mathcal{L} = f(x) + M_{\mu g}(x + \mu \alpha) - \frac{\mu}{2} \| \alpha \|^2_2.
\end{equation}
Specifically, \eqref{eq: lagrangiana ns-cmo} includes the additional term $\lambda^\top h(x)$ to account for the presence of the equality constraints. The works \cite{dhingra19} and \cite{has21} study two algorithms based on \eqref{eq: prox_aug_noeq}.

On the one hand, \cite{dhingra19} computes the saddle points of \eqref{eq: prox_aug_noeq} by means of the following primal-dual dynamics
\begin{equation}\label{eq:dhingra}
    \begin{cases} 
    \dot{x} = - \nabla f(x) - \nabla M_{\mu g}(x + \mu \alpha) \\
    \dot{\alpha} = \mu \nabla M_{\mu g}(x+ \mu \alpha ) - \mu \alpha.
    \end{cases}
\end{equation}
On the other hand, \cite{has21} observes that the choice $\alpha = -\nabla f(x)$ leads to the proximal gradient flow dynamics:
\begin{equation}\label{eq:has21}
    \dot{x} = - \nabla f(x) - \nabla M_{\mu g}(x - \mu \nabla f(x)).
\end{equation}

We observe that both approaches can be interpreted as feedback controllers governing the dual variable $\alpha$. 
In the following, we build on this control-theoretic viewpoint by designing control laws for both $\alpha$ and $\lambda$, thereby also addressing equality constraints.
\subsection{Proximal CMO}
In this section, we present the proposed method. By leveraging the proximal augmented Lagrangian as defined in ~\eqref{eq: lagrangiana ns-cmo} and inspired by CMO \cite{SIC_TAC_25}, we define a CT dynamics that, under suitably designed control, converges to the optimum of  Problem~\eqref{sec:probl_statem}.

We denote the proposed approach as Proximal CMO (Prox-CMO).

We define the CT dynamical system
\begin{equation}\label{eq:plant}
\mathcal{P}:\;
\begin{cases}
\dot{x}(t) = - \nabla f\bigl(x(t)\bigr)
              - \nabla M_{\mu g}\bigl(x(t) + \mu \alpha(t)\bigr) +\\
\qquad\quad - J_h^\top\bigl(x(t)\bigr)\lambda(t) \\[0.5ex]
y_1(t) = x(t) - \prox_{\mu g}\bigl(x(t) + \mu \alpha(t)\bigr) \\
y_2(t) = h\bigl(x(t)\bigr)
\end{cases}
\end{equation}
The state variable $x(t)$ evolves according to the gradient flow dynamics induced by the proximal augmented Lagrangian $\dot{x}(t) = - \nabla \mathcal{L}_{\mu}(x(t), \alpha(t), \lambda(t))$ that directly follows from the first-order optimality conditions.
Since Problem~\eqref{eq: problem with additional constraint} involves two constraints, we define two output signals, denoted by $y_1(t)$ and $y_2(t)$. 
Specifically, $y_2(t)\in\R^m$ derives from $h(x) = 0$, as in the CMO formulation of \cite{SIC_TAC_25}. Conversely, $y_1(t)\in\R^n$ stems from $x-z = 0$, which is replaced by the equivalent formulation $x - \prox_{\mu g}(x + \mu\alpha) = 0$ when constraining the augmented Lagrangian~\eqref{eq: lagrangiana ns-cmo0} to the manifold defined by~\eqref{eq: z star}.
The objective of the control design is to steer system~\eqref{eq:plant} to an equilibrium point while regulating both outputs to zero.

The following lemma establishes the equivalence between equilibria of~$\mathcal{P}$ and stationary points of the optimization problem.
%

\begin{lemma}
    Let $y(t) = [y_1(t)^\top, y_2(t)^\top]^\top\in\R^{n+m}$
    An equilibrium point $(x^\star,  \alpha^\star, \lambda^\star)$ of system $\mathcal{P}$ is a stationary point of problem \eqref{eq: problem with additional constraint} if and only if the equilibrium output is $y^\star = 0$.
    \label{lemma 1}
\end{lemma}

\begin{proof}
    A point $(x^\star,  \alpha^\star, \lambda^\star)$ is an equilibrium point of $\mathcal{P}$ if $\nabla f(x^\star) + \nabla M_{\mu g}(x^\star + \mu  \alpha^\star) + J_h ^\top (x)\lambda^\star = 0$.
    
    If $y^\star = 0$, then both constraints are satisfied, and the saddle point conditions~\eqref{eq: saddle point} hold. As guaranteed by Proposition~\ref{proposition 1}, $(x^\star,  \alpha^\star, \lambda^\star)$ is a stationary point of Problem \eqref{eq: problem statement}. 
    If, conversely, Eq.~\eqref{eq: saddle point} is satisfied, then $(x^\star,  \alpha^\star, \lambda^\star)$  is an equilibrium point of $\mathcal{P}$.
\end{proof}

Lemma \ref{lemma 1} relates the optimization problem to a stabilization and output regulation control problem. 
In particular, as a consequence of Lemma~\ref{lemma 1}, a stationary point of Problem \eqref{eq: problem statement} can be computed by designing suitable inputs $\alpha(t), \lambda(t)$ that drive $\mathcal{P}$ to equilibrium while also regulating the output to zero.
Accordingly, we proceed to design appropriate control laws for $\alpha(t)$ and $\lambda(t)$.

As to $\lambda(t)$, in this work we exploit the PI control law introduced in~\cite[Section III]{SIC_TAC_25}:
\begin{equation*}
    \lambda(t) = k_p h(x(t)) + k_i \int_{0}^t h(x(\tau))\, \rm{d}\tau.
\end{equation*} 
This PI control law features a straightforward structure and exhibits competitive performance in the smooth setting, as demonstrated in \cite{SIC_TAC_25}.

For $\alpha(t)$, we propose two distinct control laws. 
%


The first one is a nonlinear static state-feedback control, that represents a natural extension of the proximal gradient descent \cite{has21} to the constrained setting.
The second one is nonlinear and dynamic, and generalizes the nonsmooth PDGD \cite{dhingra19}.

%
For notational convenience, in the remainder of the paper we drop the explicit time dependence and write $x = x(t)$, $\alpha = \alpha(t)$, and $\lambda = \lambda(t)$.

\section{Prox-CMO with static feedback control of $\alpha$}
\label{sec:static}
We develop the first algorithm in the Prox-CMO family by applying the following static state-feedback controller to the plant dynamics \eqref{eq:plant}:
\begin{equation}\label{eq:static_grad_control}
    \alpha = -\nabla f(x).
\end{equation}
%


This design choice is inspired by \cite{has21}.

Using property~\eqref{eq: grad moreau envelope}, we obtain the following description of the closed-loop system defined by controller ~\eqref{eq:static_grad_control} and plant~\eqref{eq:plant}:
\begin{equation}\label{eq:prox cmo}
    \begin{cases} 
    \dot{x} = - \dfrac{1}{\mu} x + \dfrac{1}{\mu} \prox_{\mu g}(x - \mu \nabla f(x)) - J_h^\top(x) \lambda\\[0.8em]
    \dot{\lambda} = k_p J_h (x) \dot{x} + k_i h(x) 
    \end{cases} 
\end{equation}
We refer to Eq.~\eqref{eq:prox cmo} as the static Prox-CMO algorithm.

The static Prox-CMO dynamics in Eq.~\eqref{eq:prox cmo} is a natural extension of the proximal gradient flow \cite{has01} to constrained optimization problems. Indeed, the dynamics of the variables $x$ in Eq.~\eqref{eq:prox cmo} correspond to the proximal gradient flow dynamics \eqref{eq:has21} with the additional drift term $J_h^\top \lambda$ to account for the equality constraints.


\subsection{Convergence of the static Prox-CMO}\label{prox cmo global conv}

In this section, we prove global exponential stability of the static Prox-CMO under the following assumptions.
%
\begin{assumption}\label{assumption 1}
$f(x)$ is an $m_f$-strongly convex, continuously differentiable function with $L_f$-Lipschitz continuous gradient.
\end{assumption}
Under this Assumption \ref{assumption 1}, Lemma 1 in \cite{qu19} states that for any $x,x^\star \in \R^{n} $, there exists a symmetric matrix $B = B(x,x^\star) $, satisfying $ m_f I  \preceq B  \preceq L_f I $ such that: 
    \begin{equation}
        \nabla f(x) - \nabla f(x^\star) = B(x-x^\star).
    \end{equation}
\begin{assumption}\label{assumption 2}
    Function $g(x)$ is proper, lower semi-continuous, convex and non-differentiable. 
\end{assumption}

We prove the following result.
\begin{lemma}
Let $g(x)$ satisfy Assumption~\eqref{assumption 2}. Then, for any $x,x^\star \in \R^{n}$, there exists a symmetric  matrix $D =D(x,x^\star) $, satisfying $ 0  \preceq D  \preceq I $ such that:
    \begin{equation}
         \prox_{\mu g}(x) - \prox_{\mu g}(x^\star) = D\, (x-x^\star)
    \end{equation}
\end{lemma}
\begin{proof}
    Let $P = \prox_{\mu g}(x) - \prox_{\mu g}(x^\star)$, $p = x -x^\star$ and let $$D = P P^\top /(P^\top p), P \neq 0.$$ $D$ is symmetric by definition.
    Because of the nonexpansiveness of $\prox_{\mu g}$ \cite{par13}, $P^\top P \leq p^\top P$. Thus, $D\succeq0$. Also, $P^\top (P-p) = P^\top (D-I) p = p^\top D (I-D) p \leq 0$. The last inequality implies $D \preceq I$.
\end{proof}

\begin{assumption}
    \label{assumption 3}
     $h(x)$ is affine, i.e. there exists $C \in \R^{m,n}, b \in \R^m$ such that $h(x) = Cx + b $. Moreover $C$ is full row rank and there exist $ 0 < a_1 \leq a_2 $ such that  $  a_1 I \preceq CC^\top \preceq a_2 I $.
\end{assumption} 

Given the matrices $B$ and $D$, we now introduce a matrix $Z$, which is instrumental in the convergence analysis of~\eqref{eq:prox cmo}:
\begin{equation}\label{eq: def Z}
    Z(x) = Z \doteq \frac{1}{\mu}(I - D) + D B .
\end{equation}
The following two lemmas characterize the properties of $Z$.
\begin{lemma}\label{lemm:Z}
    Let Assumption~\ref{assumption 1} hold. If $D$ is diagonal with all entries satisfying $0 \leq D_{ii} \leq 1$ and $\mu \leq \frac{1}{L_f}$, then
    \begin{equation}
        Z + Z^\top \succeq \frac{3}{2}B.
    \end{equation}
\end{lemma}
\begin{proof}
    The result immediately follows from Lemma 6 in \cite{qu19}, after noting that Assumption~\ref{assumption 1} implies the existence of a matrix $A_B$ which is a Cholesky factor of $B$, i.e., $B = A_B A_B^\top$.
\end{proof}
\begin{lemma}\label{lemm:ZZtop}
    Let Assumption~\ref{assumption 1} hold. Then
    \begin{equation}
        Z Z^\top \preceq \left(L_f+\frac{1}{\mu}\right)^2 I.
    \end{equation}
\end{lemma}
\begin{proof}
    Let us expand $Z Z^\top$: 
    \begin{equation}
        Z Z^\top = \frac{1}{\mu}(I-D)BD + \frac{1}{\mu^2}(I-D)^2 + DB^2D + \frac{1}{\mu}DB(I-D)
    \end{equation}
    We bound each term as: $D B^2 D \preceq B^2 \preceq L_f^2 I$,~ $\frac{1}{\mu}(I-D) B D \preceq \frac{1}{\mu} B \preceq \frac{1}{\mu }L_f I$, and $\frac{1}{\mu^2}(I-D)^2 \preceq \frac{1}{\mu^2} I$.
    Then, 
    \begin{equation}
        Z Z^\top \preceq \left(\frac{1}{\mu^2} + \frac{2}{\mu}L_f + L_f^2\right) I
    \end{equation}
    and the result follows from the collection of the squares.
\end{proof}

Let  $\omega =  [x^\top, \lambda^\top]^\top$ be the state vector of \eqref{eq:prox cmo}. The static Prox-CMO dynamics can be rewritten as
\begin{equation}
    \dot{\omega} = F(\omega)
\end{equation}
We denote as $\omega^\star = [x^{\star \top},\lambda^{\star \top}] ^\top $ the equilibrium point of the closed-loop system \eqref{eq:prox cmo}, i.e., the point satisfying $ F(\omega^\star) = 0$.
We now state the main result of this section.
\begin{theorem}
    Let Assumptions~\ref{assumption 1},\ref{assumption 2} and \ref{assumption 3} hold. Then, given an arbitrary $k_i > 0$, a positive real $\varepsilon$ satisfying 
    \begin{equation}
         \varepsilon < \frac{3 m_f}{4\left(L_f + \frac{1}{\mu}\right) - 3 m_f} < 1,
    \end{equation} 
    and $k_p > 0$ 
    \begin{equation}\label{th1 cond on kp}
         k_p = \varepsilon \frac{k_i}{\left(L_f+\frac{1}{\mu}\right)},
    \end{equation}
    the static Prox-CMO dynamics~\eqref{eq:prox cmo} is globally exponentially stable with rate 
    \begin{equation}\label{th1 conv rate}
        r = \min \left(\frac{3}{2} m_f \frac{1+\varepsilon}{1-\varepsilon} -  2\, \frac{\varepsilon}{1-\varepsilon}\left(L_f + \frac{1}{\mu}\right) , k_p a_1 \right) > 0,
    \end{equation}
    i.e., there exist  $c \in \R_+$ such that:
    
    \begin{equation}    
        \| \omega(t) -\omega^\star\|^2_2 \leq c\, e^{-\frac{1}{2} r t},  
    \end{equation}  
\end{theorem}

\begin{proof}
    Under the stated assumptions, \eqref{eq:prox cmo} can be equivalently rewritten as 
    \begin{equation}
        \dot{\omega} = F(\omega) - F(\omega^\star) = G(x) ~[\omega - \omega^\star]
    \end{equation}
    where $G(x) \in \R^{n+m,n+m}$ is given by
    \begin{equation*}
        G(x) =
        \begin{bmatrix}
            -Z & -C^\top \\
            - k_p CZ + k_i C & -k_p CC^\top
        \end{bmatrix},
    \end{equation*}
    and the matrix $Z = Z(x)$ is defined in Eq.~\eqref{eq: def Z}. \\
    We consider the quadratic Lyapunov function:
    \begin{equation}
        V(\omega) = (\omega -\omega^\star)^\top P (\omega -\omega^\star)
    \end{equation}
    with
    \begin{equation}
        P = 
    \begin{bmatrix}
        \rho I & 0 \\
        0 & I 
    \end{bmatrix} \succ 0.
    \end{equation}
    A sufficient condition for global exponential stability with rate $r$ is:
    \begin{equation}
        \dot{V}(\omega) =  (\omega -\omega^\star)^\top (G^\top P + P G) (\omega -\omega^\star)\leq - r V(\omega).
        \label{eq: condition for prox cmo}
    \end{equation}   
    Let $ Q \doteq -(G^\top P + P G + r P) \succeq 0$. Then, condition \eqref{eq: condition for prox cmo} is equivalent to requiring $Q \succeq 0$. After performing the matrix multiplications, condition $Q \succeq 0$ can be explicitly written as
    \begin{equation}\label{th1: Qsucc0}
       Q = \begin{bmatrix}
            \rho Z + \rho Z^\top - r \rho I & \star\\
            k_p CZ + \rho C - k_i C & 2k_p CC^\top - r I
        \end{bmatrix} \succeq 0.
    \end{equation} 
    Observe that, for $r \leq k_p a_1$, a sufficient condition for \eqref{th1: Qsucc0} is
    \begin{equation}\label{th1: Q1succ0}
       Q' \doteq \begin{bmatrix}
            \rho Z + \rho Z^\top - r \rho I& \star\\
            k_p CZ + \rho C - k_i C& k_p CC^\top
        \end{bmatrix} \succeq 0.
    \end{equation} 
    We now employ the Schur complement to derive conditions under which \eqref{th1: Q1succ0} holds. Since $k_p CC^\top \succeq 0$ for $k_p>0$, the Schur complement condition $Q'/(k_p CC^\top) \succeq 0$ takes the form
    \begin{equation}
        \begin{aligned}
        & \rho( Z+Z^\top) - \rho r I - [k_p Z^\top  + (\rho -k_i)I]\\
        & \qquad C^\top(k_p CC^\top)^{-1} C [k_p Z  + (\rho -k_i)I] \succeq 0.\label{th1: schur1}
        \end{aligned}
    \end{equation}
    Since $CC^\top$ is invertible, it holds that $C^\top(CC^\top)^{-1}C \preceq I$.  Therefore, a sufficient condition for \eqref{th1: schur1} is:
    \begin{equation}\label{th1: schur2}
        \begin{aligned}
        & \rho(Z+Z^\top) - \rho r I - k_p Z^\top Z + \\
        & - (\rho - k_i) (Z + Z^\top) - \frac{1}{k_p} (\rho - k_i)^2I \succeq 0
        \end{aligned}
    \end{equation}
    where the matrix products have been explicitly expanded.
    
    We now analyze the term
    \begin{equation}
        k_p^2 Z^\top Z
    + k_p (\rho - k_i)(Z + Z^\top)
    + (\rho - k_i)^2 I.
    \end{equation}
    Depending on the sign of $\rho - k_i$, two different upper bounds can be obtained.
    If $\rho - k_i < 0$, then
    \begin{equation}\label{th1 upper bound 1}
        \begin{aligned}
             & k_p^2 Z^\top Z  
             + k_p(\rho - k_i)(Z + Z^\top) 
             + (\rho - k_i)^2 I 
             \preceq \\[3pt]
             & \left[ 
                k_p^2 \left( \frac{1}{\mu} + L_f \right)^2
                - 3 k_p (k_i - \rho) m_f
                + (\rho - k_i)^2 
               \right] I
        \end{aligned}
    \end{equation}
    where Lemmas~\ref{lemm:Z} and~\ref{lemm:ZZtop} have been used.     
    Conversely, when $\rho - k_i \geq 0$, the following upper bound holds:
    \begin{equation}\label{th1 upper bound 2}
        \begin{aligned}
             & k_p^2 Z^\top Z  
             + k_p(\rho - k_i)(Z + Z^\top) 
             + (\rho - k_i)^2 I 
             \preceq \\[3pt]
             & \left[ 
                k_p \left( \frac{1}{\mu} + L_f \right) + (\rho - k_i)\right]^2 I.
        \end{aligned}
    \end{equation}
    
    The bound in \eqref{th1 upper bound 2} is more conservative than that in \eqref{th1 upper bound 1}. For this reason, we proceed under the assumption $\rho - k_i < 0$.
    By completing the square, \eqref{th1 upper bound 1} is further bounded by
    \begin{equation}\label{th1 upper bound 3}
        \begin{aligned}
             & k_p^2 Z^\top Z  
             + k_p(\rho - k_i)(Z + Z^\top) 
             + (\rho - k_i)^2 I 
             \preceq \\[3pt]
             & \left[ k_p\left(L_f + \frac{1}{\mu}\right) + (\rho - k_i) \right]^2 I \\
                & + k_p(k_i - \rho)\left[ 2\left(L_f + \frac{1}{\mu}\right) - 3m_f \right] I.
        \end{aligned}
    \end{equation}
    
    Setting the Lyapunov function parameter as $\rho = k_i - k_p\left(L_f + \frac{1}{\mu}\right)$, inequality \eqref{th1 upper bound 3} reduces to
    \begin{equation}\label{th1 upper bound 4}
        \begin{aligned}
             & k_p^2 Z^\top Z  
             + k_p(\rho - k_i)(Z + Z^\top) 
             + (\rho - k_i)^2 I 
             \preceq \\
             & 2k_p^2\left(L_f + \frac{1}{\mu}\right)^2 - 3k_p^2\left(L_f + \frac{1} {\mu}\right)m_f.
        \end{aligned}
    \end{equation}
    Note that the considered choice for $\rho$ is a valid positive parameter, as by the condition in Eq.~\eqref{th1 cond on kp} we obtain $\rho = (1-\epsilon)k_i > 0$.
    Substituting Eq.~\eqref{th1 upper bound 4} and the selected value of $\rho$ into \eqref{th1: schur2} yields the following sufficient condition:
    \begin{equation}\label{eq: th1_bnd_4b}
        \begin{aligned}
            \left( \frac{3}{2} m_f k_i - r \right) \left[k_i - k_p\left(L_f + \frac{1}{\mu}\right) \right] - \\
            - 2k_p \left(L_f + \frac{1}{\mu}\right)^2 + 3k_p \left(L_f + \frac{1}{\mu}\right)m_f \geq 0
        \end{aligned}        
    \end{equation}
    Replacing Eq.~\eqref{th1 cond on kp} in Eq.~\eqref{eq: th1_bnd_4b}, we obtain
    \begin{equation}
        \frac{3}{2} m_f k_i (1+\varepsilon) - 2\varepsilon k_i\left(L_f + \frac{1}{\mu}\right) - r\, k_i(1-\varepsilon) \geq 0,
    \end{equation}
    which, solved for the convergence rate $r$, gives
    \begin{equation}
        r < \frac{3}{2} m_f \frac{1+\varepsilon}{1-\varepsilon} -  2\, \frac{\varepsilon}{1-\varepsilon}\left(L_f + \frac{1}{\mu}\right).
    \end{equation}
    Since $r$ must be positive, the parameter $\varepsilon$ must satisfy
    \begin{equation}
        \varepsilon < \frac{3 m_f}{4\left(L_f + \frac{1}{\mu}\right) - 3 m_f}
    \end{equation}
    Note that this upper bound on $\varepsilon$ is always positive and strictly smaller than~$1$, by the definitions of $L_f$ and $m_f$.   
\end{proof}

\begin{remark}
    From~\eqref{th1 conv rate}, we see that the algorithm’s convergence speed increases with larger values of $k_p$.
    However, due to condition~\eqref{th1 cond on kp}, $k_p$ is upper bounded by a value that depends on $m_f$, $L_f$, and $\mu$. While the first two parameters are fixed and depend on the cost function, $\mu$ is a free parameter. Ideally, large values of $\mu$ can increase convergence speed.
    However, we observe that in some practical applications, large values of $\mu$ lead to numerical difficulties. This implies selecting a smaller $k_p$ and, thus, slower convergence. 
    In this sense, the upper bound on $k_p$ that guarantees convergence may be conservative. Nonetheless, selecting $k_p$ above the guaranteed bound can accelerate convergence in practice, without losing stability. Moreover, compared to mere integral action, even a small $k_p$ enhances the convergence rate.
\end{remark}

\subsection{Comparison with PI-PGD}\label{sec: comp with pipgd}

As previously discussed, we can interpret the dynamics in \eqref{eq:prox cmo} as an extension of proximal gradient descent to constrained optimization problems. A related approach is presented in \cite{cent25}, where the authors address the same class of problems and propose a CMO-based dynamics referred to as PI-PGD:
\[
\begin{cases}
\dot{x} = -x + \prox_{\gamma g} \left( x - \gamma \left( \nabla f(x) + J_h^\top(x) \lambda \right) \right) \\[0.4em]
\dot{\lambda} = k_p J_h(x) \dot{x} + k_i h(x)
\end{cases}
\]
Although both algorithms aim to solve the same composite and constrained optimization problem \eqref{eq: problem statement}, the method in \cite{cent25} is derived from the standard Lagrangian $\mathcal{L}(x, \lambda) = f(x) + g(x) + \lambda^\top h(x)$. 

In \cite[Lemma 1]{cent25}, the authors show that the differential inclusion arising from the first-order necessary conditions~\eqref{foc} naturally leads to the definition of the proximal operator, without resorting to the Moreau envelope. As a consequence, the Lagrange multiplier~$\alpha$, which in our formulation enables the design of distinct control strategies acting on the non-differentiable component, does not appear in their framework. The introduction of the additional degree of freedom represented by~$\alpha$ renders the Prox-CMO approach more flexible and general.

Finally, we observe that the main structural difference between the PI-PGD dynamics and Eq.~\eqref{eq:prox cmo} lies in the placement of the term $-J_h(x)^\top \lambda$. In Eq.~\eqref{eq:prox cmo} it appears outside the proximal operator. This feature, stemming from the Moreau-based formulation, allows for a clearer separation between the constraints and the nonsmooth component of the objective function.
\section{Prox-CMO with dynamic feedback control of $\alpha$}
\label{sec:dynamic}

In this section, we introduce and analyze the second algorithm in the Prox-CMO family. 
In~\cite{dhingra19}, the primal-dual gradient dynamics derived from the proximal augmented Lagrangian implements an integral action on the dual variable $\alpha$ in the form $\dot{\alpha} = \mu (M_{\mu g}(x + \mu \alpha) - \alpha)$. 
A natural extension to improve the convergence speed is to introduce an additional proportional action to the integral one.
However, an exact PI control law for~$\alpha$, driven by the output~$y_1$ defined in~\eqref{eq:plant}, would result in a discontinuous control law, complicating the subsequent analysis and potentially causing numerical issues when conducting the closed-loop system simulation; see, e.g., \cite{filippov1960differential} for a detailed treatment of differential equations with a discontinuous right-hand side. To see how such a discontinuity arises, consider the PI control law
\[\alpha = k_p y_1(x,\alpha) + k_i \int_{0}^{t} y_1(x,\alpha)\,d\tau.\] 
Differentiating $\alpha$ with respect to time yields
\[\dot{\alpha} = k_p[\dot{x} - \frac{\partial}{\partial{v}} \prox_{\mu g}(v) (\dot{x} + \mu \dot{\alpha}) ] + k_i (x - \prox_{\mu g}(x + \mu \alpha)).\]
We note that the term
$
\frac{\partial}{\partial v}\prox_{\mu g}(v)(\dot{x} + \mu \dot{\alpha})
$
can be at best discontinuous and, in general, has no closed-form expression.

To circumvent this issue, we propose the following modified dynamic controller for the dual variable:
\begin{equation} \label{eq: alpha controller 2}   
\begin{aligned} 
    &\dot{\alpha} = k_1 (\nabla f(x) +J_h^\top \lambda) + k_2 \alpha + k_3 \nabla M_{\mu g}(x+ \mu \alpha ),
\end{aligned}
\end{equation}
where $k_1,k_2,k_3 \in \mathbb{R}$ are design gains.
Eq.~\eqref{eq: alpha controller 2} extends the solely integral action of \eqref{eq:dhingra}, which is obtained for the special choices $k_3 = -k_2 = \mu$ and $k_1 = 0$.
The closed-loop dynamics arising from the application of the nonlinear dynamic controller~\eqref{eq: alpha controller 2} to the plant in Eq.~\eqref{eq:plant}, hereinafter referred to as dynamic Prox-CMO, is given by:
\begin{equation}
    \begin{cases} 
    \dot{x} = - \nabla f(x) - \nabla M_{\mu g}(x + \mu \alpha) -  J_h(x)^\top {\lambda} \\[3pt]
    \dot{\alpha} = k_1 (\nabla f(x) +J_h^\top \lambda) + k_2 \alpha + k_3 \nabla M_{\mu g}(x+ \mu \alpha ) \\[3pt]
    \dot{\lambda} = k_p J_h (x) \dot{x} + k_i h(x).
    \end{cases}
    \label{eq: alpha cmo}
\end{equation}

Motivated by the relevance of the setting defined by considering the special case when no equality constraints are present (see, e.g., \cite{dhingra19,ding19}), we start by considering the case of the unconstrained optimization problem 
\begin{equation}\label{eq:unconstrained_composite}
    \min_{x \in \R^{n}} ~ f(x) + g(x).
\end{equation}

The application of the dynamic Prox-CMO framework to solve \eqref{eq:unconstrained_composite} leads to the simplified dynamics obtained by discarding the constraints-related terms:
\begin{equation}
    \begin{cases} 
    \dot{x} = - \nabla f(x) - \nabla M_{\mu g}(x + \mu \alpha) \\
    \dot{\alpha} = k_1 \nabla f(x) + k_2 \alpha + k_3 \nabla M_{\mu g}(x+ \mu \alpha ).
    \end{cases}
    \label{eq: alpha cmo simplified dynamics}
\end{equation}

\subsection{Convergence of dynamic Prox-CMO: the case of unconstrained problems}
\label{sec:analysis_dynamic_proxpicmo}
Following the approach used for the static Prox-CMO analysis, we rewrite system~\eqref{eq: alpha cmo simplified dynamics} in the form
\begin{equation}
    \dot{\zeta} = F(\zeta),
\end{equation}
where $\zeta =  [x^\top, \alpha^\top]^\top$ represents the state vector. 

Under assumptions~\ref{assumption 1} and \ref{assumption 2}, Problem~\eqref{eq:unconstrained_composite} is strongly convex and thus admits a unique global optimum, denoted by $\zeta^\star = [x^{\star \top},\alpha^{\star \top}] ^\top $. As a consequence, the dynamics \eqref{eq: alpha cmo} admit a unique equilibrium point. In the following theorem, we establish its stability.

\begin{theorem} \label{theorem 3}

Let Assumptions~\ref{assumption 1} and \ref{assumption 2} hold. Then, given arbitrary $k_1,k_3 > 0$,
    \begin{equation}\label{def k2 crit}
         k_{2}^{\rm{crit}} \doteq-k_3 - \frac{k_1 ^2 \mu }{2 k_3} \frac{L_f^2}{m_f} < 0,
    \end{equation}
    and $k_2 <  k_{2}^{\rm{crit}}$, the dynamics~\eqref{eq: alpha cmo simplified dynamics} is globally exponentially stable with rate
    \begin{equation}\label{th3: conv rate}
    \begin{aligned}
        r = \min \left( m_f, -2(k_2- k_{2}^{\rm{crit}} ) \right) > 0,
    \end{aligned}
    \end{equation}
    i.e., there exist  $c \in \R_+$ such that:
    \begin{equation}
    \| \zeta(t) - \zeta^\star\|^2 \leq c e^{-\frac{1}{2} r t}.
    \end{equation}  
\end{theorem}

\begin{proof}

    Under the stated assumptions, we can rewrite the closed-loop dynamics~\eqref{eq: alpha cmo simplified dynamics} in the compact form as $\dot{\zeta} = F(\zeta) = G(x) [\zeta - \zeta^\star]$ where
    \begin{equation}
        G(x) = \begin{bmatrix}
            - T(x) & -U(x)\\
             k_1 B(x) + \frac{k_3}{\mu} U(x) & k_2 I + k_3 U(x)
        \end{bmatrix}.
    \end{equation}
    We have introduced the matrices 
    \begin{equation}
        U(x) = U \doteq I - D, \quad T(x) = T \doteq B+\frac{1}{\mu}U.
    \end{equation}
        
    To analyze stability, we consider the quadratic Lyapunov function
    \begin{equation}
        V(\zeta) = (\zeta - \zeta^\star) P (\zeta - \zeta^\star)^\top, 
    \end{equation}
    with
    \begin{equation}
        P = \begin{bmatrix}
                \frac{k_3}{\mu} I & 0 \\ 0 & I
            \end{bmatrix}.
    \end{equation}
    Note that $P \succ 0$ for all $k_3 > 0$, ensuring that $V$ is positive definite.
    
    A sufficient condition for global exponential stability of the unique equilibrium $\zeta^\star$ is the existence of a constant $r > 0$ such that
    \begin{equation}\label{th2: lyap_decay}
        \dot{V}(\zeta) =  (\zeta -\zeta^\star)^\top (G^\top P + P G) (\zeta -\zeta^\star)\leq - r V(\zeta).
    \end{equation}   
    
    Condition \eqref{th2: lyap_decay} is equivalent to the linear matrix inequality
    \begin{equation}\label{th3: q succ 0}
        Q \doteq -(G^\top P + P G + r P) \succeq 0.
    \end{equation}
    By explicitly computing the matrix $Q = Q^\top$, we obtain
    \begin{equation}
       Q = \begin{bmatrix}
            2 \frac{k_3}{\mu} T - r \frac{k_3}{\mu}I& \star\\
           - k_1 B & -2k_2 I - 2k_3 U -rI
        \end{bmatrix} \succeq 0.
    \end{equation}
    If $r \leq m_f$, a sufficient condition for \eqref{th3: q succ 0} is given by
    \begin{equation}\label{eq:th2_cond_Qprime}
       Q' = \begin{bmatrix}
             \frac{k_3}{\mu} T & \star\\[3pt]
           - k_1 B & -2k_2 I - 2k_3 U -rI  
        \end{bmatrix} \succeq 0.
    \end{equation}
    We apply Schur complement to verify the condition in Eq.~\eqref{eq:th2_cond_Qprime}, thus obtaining the conditions $T \succ 0$ and $Q'/(k_3 \mu^{-1} T) \succeq 0$. The first condition is trivially verified by the property $B \succ 0$ and by the definition of $T$. The latter condition is expanded as
    \begin{equation}
       -2k_2 I - 2k_3 U -rI - (-k_1 B ) \left(\frac{k_3}{\mu} T \right)^{-1}(-k_1 B)\succeq 0
        \label{eq: schur condition for alpha cmo simplified}
    \end{equation}
    Since $B \preceq L_f I$ and $T^{-1} \preceq B^{-1} \preceq \frac{1}{m_f} $, a sufficient condition for \eqref{eq: schur condition for alpha cmo simplified} is
    \begin{equation}
        -2 k_2 - 2 k_3 - r  - \frac{k_1^2 \mu}{k_3} \frac{L_f^2}{m_f}  \geq 0.
    \end{equation}
    Solving for $r$ yields $r \leq   -2(k_2- k_{2}^{\rm{crit}} )$.
    Under the assumed conditions on the gain $k_2$, the resulting decay rate $r$ is strictly positive, which concludes the proof. 
\end{proof}
\begin{remark}
    The works \cite{dhingra19} and \cite{ding19} address a slightly more general unconstrained problem of the form $\min_{x \in \R^{n}} f(x) + g(Tx)$,
    where the inclusion of a linear transformation $T$  broadens the method's applicability to additional settings, such as distributed implementation.

    Our framework can handle this case by introducing
    a suitable function $h(x)$ and additional auxiliary variables, as illustrated in the system identification example in Sec. \ref{sec: set memb}.
    However, an explicit extension to the general case with a linear transformation of the parameters will be addressed in future work.
    %
\end{remark}
%


\subsection{Convergence of dynamic Prox-CMO: the case of constrained problems}
\label{sec:analysis_dynamic_proxpicmo_cns}

We now analyze the convergence of the dynamic Prox-CMO dynamics Eq.~\eqref{eq: alpha cmo} under Assumptions \ref{assumption 1}, \ref{assumption 2}, and \ref{assumption 3} stated in Section \ref{prox cmo global conv}.

We can rewrite the closed-loop dynamics \eqref{eq: alpha cmo} in the form $\dot{\eta} = F(\eta)$, with $\eta =  [x^\top, \alpha^\top, \lambda^\top]^\top$ state vector and $\eta^\star = [x^{\star \top},\alpha^{\star \top}, \lambda^{\star \top}] ^\top $ equilibrium point of the closed-loop system \eqref{eq: alpha cmo}, satisfying $F(\eta^\star) = 0$.

\begin{theorem}
    Let Assumptions \ref{assumption 1},\ref{assumption 2} and \ref{assumption 3} hold. 
    Given $k_1, k_3, k_i, k_p, \gamma > 0$, and $k_2, k_2^{crit} < 0$ defined as in \eqref{def k2 crit}, a suitably chosen constant $\varepsilon < -k_2 (1/ \mu + L_f)^2$, and 
    \begin{align}
        & \delta  < \left( \frac{2 m_f k_1}{\mu} + \frac{k_1^2}{\mu k_2} \right),\label{eq: cond on delta} \\
        & k_2 <\min \left( k_2^{crit}, -2\frac{k_1^2}{\gamma k_p} -2 k_p \gamma \right),
    \end{align}
    the dynamics in equation~\eqref{eq: alpha cmo} is globally exponentially stable with rate 
    \begin{equation}
    r = \min \left(\!-2(k_2\!-\!k_{2}^{\rm{crit}}),k_p a_1, -\frac{k_2}{2},2m_f + \frac{k_1}{k_2}\! -\! \frac{\mu \delta}{k_1}\right)
    \end{equation}
    i.e., there exist  $c \in \R_+$ such that:
    \begin{equation}
        \| \eta(t) -\eta^\star\|^2 \leq c e^{-\frac{1}{2} r t}. 
    \end{equation}
\end{theorem}

\begin{proof}
    We rewrite the dynamics \eqref{eq: alpha cmo} as $\dot{\eta} = F(\eta) - F(\eta^\star) = G(x) [\eta - \eta^\star]$. By defining the matrices $U(x) = U \doteq I - D(x)$ and $T(x) = T \doteq B +\frac{1}{\mu}U$, the matrix $G(x)$ is given by
    \begin{equation}
    G =
        \begin{bmatrix}
           -T  & - U &  -C^\top \\
           k_1 B + \frac{k_3}{\mu}(I-D) & k_2 I +k_3 U & k_1 C^\top \\
           -k_pCT + k_iC & -k_pCu & -k_{p}CC^\top
        \end{bmatrix}
    \end{equation}
    where all dependencies on $x$ are omitted for brevity. 
    
    Consider the quadratic candidate Lyapunov function 
    \begin{equation}
        V(\omega) = (\omega -\omega^\star)^\top P (\omega -\omega^\star), \quad 
        P = \begin{bmatrix}
            \frac{k_3}{\mu} I & 0 & 0 \\
            0 & I & 0 \\
            0 & 0 & \gamma I
        \end{bmatrix}.
    \end{equation}
    A sufficient condition for global exponential stability is the existence of $r > 0$ such that
    \begin{equation}
        \dot{V}(\omega) =  (\omega -\omega^\star)^\top (G^\top P + P G) (\omega -\omega^\star)\leq - r V(\omega),
    \end{equation}   
    which is equivalent to $Q \doteq -(G^\top P + P G + r P) \succeq 0$. The matrix $Q$ admits the block decomposition
    \begin{equation}
       Q = 
        \begin{bmatrix}
            Q_1 & \star & \star \\
            Q_4 & Q_2 & \star \\    
            Q_6 & Q_5 & Q_3 
        \end{bmatrix}.
    \end{equation}
    where
    \begin{align*}
    Q_1 &= \frac{2k_3}{\mu}T - \frac{k_3}{\mu}rI 
    &\quad Q_2 &= - 2k_2I - 2k_3 U - rI, \\
    Q_3 &= 2 \gamma k_p CC^\top - r \gamma I 
    &\quad Q_4 &= -k_1 B, \\
    Q_5 &= \gamma k_p C U - k_1 C 
    &\quad Q_6 &= C \left( \frac{k_3}{\mu } I - k_i \gamma I + k_p \gamma T \right).
    \end{align*}

    To simplify the analysis, define the auxiliary matrix $Q'$ obtained from $Q$ by replacing $Q_3$ with
    \begin{equation}
        Q_3' = \gamma k_p CC^\top.
    \end{equation}
    Since $Q \succeq Q'$ holds for $r \leq k_p a_1$, $Q' \succeq 0$ is a sufficient condition for $Q \succeq 0$. In what follows, we resort to the Schur complement argument twice to prove $Q' \succeq 0$, as it consists of a $3 \times 3$ block matrix.
    
    A first necessary condition for $Q' \succeq 0$ is
    \begin{equation}
        \begin{bmatrix}
            Q_1 & \star \\
            Q_4 & Q_2
        \end{bmatrix} \succeq 0,
        \label{th4: schur condition on first block}
    \end{equation}
    which coincides with the condition analyzed in Theorem~\ref{theorem 3} and is therefore satisfied for all $r$ in \eqref{th3: conv rate}.
    The second Schur complement condition is
    \begin{equation}\label{th4: schur 1}
        \begin{bmatrix}
            Q_1 & \star \\
            Q_4 & Q_2
        \end{bmatrix} -
        \begin{bmatrix}
            Q_6^\top \\
            Q_5^\top
        \end{bmatrix} (Q_3')^{-1}
        \begin{bmatrix}
            Q_6 & Q_5
        \end{bmatrix} \succeq 0.
    \end{equation}
    Choosing $k_3/ \mu = \gamma k_i$ and using the inequality $C^\top(CC^\top)^{-1}C \preceq I$ (which holds because $CC^\top$ is invertible), condition \eqref{th4: schur 1} is rewritten as 
    \begin{equation}
        X \doteq \begin{bmatrix}
            M & \star \\
            N & R
        \end{bmatrix} \succeq 0.
    \end{equation}
    with
    \begin{subequations}
    \begin{align}
        M &=  \gamma k_i (2T - r I) -\gamma k_p T^2 \\
        N &=  \frac{k_1}{\mu} U - k_p \gamma UT \\
        R &= -2 k_2 I - r I - 2 \mu \gamma k_i U - \frac{1}{\gamma k_p}(k_1I - k_p \gamma U)^2. \label{elemento R della matrice schur2}\\
        &= -2 k_2 I - r I - 2 \mu \gamma k_i U - \gamma k_p U^2 + 2k_1 U - \frac{k_1^2}{\gamma k_p}U \notag
    \end{align}
    \end{subequations}
    %
    Choosing $\mu \gamma k_i = k_1 $, we observe that 
    \begin{equation*}
        -2 k_2 I - r I  - \gamma k_p U^2  - \frac{k_1^2}{\gamma k_p}U \succeq - k_2 I
    \end{equation*}
    for $r < -\frac{k_2}{2}$ and $ -k_2 > 2 [k_1^2 / (\gamma k_p) + k_p \gamma]$. Thus $X \succeq 0$ holds if $X'  \succeq 0$ where $X'$ is defined by replacing $R$ with $R' \doteq -k_2 I$. We prove $X' \succeq 0$ by Schur complement, i.e., $R' \succeq 0$ (which holds since $k_2 < 0$) and $X/R' \succeq 0$ which is explicitly expanded as:
    \begin{equation}\label{eq:shur_cond2_th3}
        \begin{aligned}
        &2 \frac{k_1}{\mu} T - r\frac{k_1}{\mu} I -\gamma k_p T^2 + \\
        &+ \left( \frac{k_1}{\mu} U - k_p \gamma UT\right)^\top \frac{1}{k_2} \left( \frac{k_1}{\mu} U - k_p \gamma UT \right) \succeq 0.
        \end{aligned}
    \end{equation}
    
    Exploiting the property $0 \preceq U \preceq I$ and the bounds on the eigenvalues of $T$, a sufficient condition for Eq.~\eqref{eq:shur_cond2_th3} is
    \begin{align}\label{eq:shur_cond3_th3}
        &2 \frac{k_1}{\mu} T - r \frac{k_1}{\mu} I -\gamma k_p T^2 + \frac{1}{k_2}\left(\frac{k_1}{\mu} I - k_p \gamma T\right)^2 \succeq \\
        & \quad\frac{k_1}{\mu} (2 m_f - r) I - \gamma k_p \left( L_f + \frac{1}{\mu} \right)^2 I +\\
        & \qquad + \frac{1}{k_2} \left[ \frac{k_1^2}{\mu^2} + k_p^2 \gamma^2  \left(  L_f + \frac{1}{\mu} \right)^2 \right] I \succeq 0.
    \end{align}
    Choosing $\gamma k_p = \varepsilon / \left( L_f + \frac{1}{\mu} \right)^2 $ with $\varepsilon$ defined in the statement of the theorem, the condition \eqref{eq:shur_cond3_th3} is equivalent to:
    \begin{align}    \label{eq:scalar_bnd_shur2_th3_a}
        &\frac{k_1}{\mu} (2 m_f - r)  - \varepsilon + \frac{1}{k_2} \left[\frac{k_1^2}{\mu^2} + \frac{\varepsilon^2 }{\left( L_f + \frac{1}{\mu} \right)^2}\right] \geq 0.
    \end{align}
    Defining $\delta = \varepsilon - \frac{1}{k_2} \frac{\varepsilon^2}{\left( L_f + \frac{1}{\mu} \right)^2}$, which is positive under the considered assumption on $\epsilon$, Eq.~\eqref{eq:scalar_bnd_shur2_th3_a} leads to the condition:
    \begin{equation}
        \frac{k_1}{\mu}  (2 m_f - r) + \frac{k_1^2}{k_2 \mu} - \delta  \geq 0,
    \end{equation}    
    that yields the following bound on $r$:
    \begin{equation}
        r \leq 2 m_f + \frac{k_1}{k_2} - \frac{\delta \mu}{ k_1}.
    \end{equation}
    Imposing $r>0$ leads to the condition in Eq.~\eqref{eq: cond on delta}. This concludes the proof.
\end{proof}

\section{Numerical results \label{simulations}}\label{sec:num_es}
In this section, we propose  numerical results that illustrate the effectiveness of the proposed Prox-CMO approach in different frameworks.
For all numerical experiments we use MATLAB R2025b on a processor i9-13900K with 64 GB of DDR5 RAM.

\subsection{Unbiased Lasso}
Lasso \cite{tib96} consists of a least-squares problem and $\ell_1$ regularization to promote sparse solutions. If the cost function $f(x)$ is strongly convex and has a unique sparse minimizer, the $\ell_1$ regularization is not necessary, but it enhances the convergence speed of proximal gradient-based algorithms, at the price of a biased solution; see, e.g.,  \cite{SIC_IISTA_25} for details.
As studied in \cite{SIC_IISTA_25}, we can eliminate the bias without sacrificing sparsity by enforcing the first-order optimality condition \(\nabla f(x) = 0\). Specifically, we consider the following constrained version of Lasso, which is of the form \eqref{eq: problem statement}:
\begin{equation}\label{bl lasso}
    \begin{aligned}
        \min_{x \in \R^{n}} \quad & \frac{1}{2}  \| Ax - b \|_2^2 + \rho \| x \|_1  \\
        \textrm{s.t.} \quad &  A^\top (Ax-b)  = 0
    \end{aligned}
\end{equation}
where $A \in \mathbb{R}^{m \times n}$, $m\geq n$, $b \in \mathbb{R}^m$, and $\rho > 0$.

In this setting, the dynamic Prox-CMO equations~\eqref{eq: alpha cmo} are 
\begin{equation*}
    \begin{cases} 
    \dot{x} =& \!- \!A^\top(Ax -b) \!- \! \sign(x + \mu \alpha) \min \left( \frac{|x + \mu \alpha|}{\mu}, 1 \right) \!-\!  A A^\top \lambda \\[3pt]
    \dot{\alpha} =& k_1 (A^\top(Ax -b) +A A^\top \lambda) + k_2 \alpha + \\[3pt]
    &k_3 \sign(x + \mu \alpha) \min \left( \frac{|x + \mu \alpha|}{\mu}, 1 \right) \\[3pt]
    \dot{\lambda} =& k_p A A^\top \dot{x} + k_i A^\top (Ax-b) .
    \end{cases}
\end{equation*}

We compare the performance of dynamic Prox-CMO with integral ISTA (I-ISTA) proposed in~\cite{SIC_IISTA_25} and PI-PGD~\cite{cent25} for solving Problem~\eqref{bl lasso}.

We consider  $n=100$, $m=110$, $\| x \|_0 = 20$ and we set $\rho =1$. We randomly generate the non-zero components of $x$ from a uniform distribution in $[-1,-0.5] \, \cup \, [0.5, 1]$ and the components of  $A$ from a Gaussian distribution $\mathcal{N}(0, \frac{1}{m})$.

We set $\mu = 0.5, k_1 = -10, k_2 = -1 , k_3 = -9, k_i = 0.8, k_p = 1$ for dynamic Prox-CMO. For PI-PGD, we consider $\gamma = 1/L$, see \cite[Theorem 3]{cent25}, and $k_p = k_i = 20$.
For I-ISTA, we set $k_i=10^{-3}$ and $\alpha=0.05$.
We integrate the CT algorithms using MATLAB \texttt{ode15s} solver over the interval $[0,10^3]$.

\begin{figure}[ht]
    \centering
    \includegraphics[width=\linewidth]{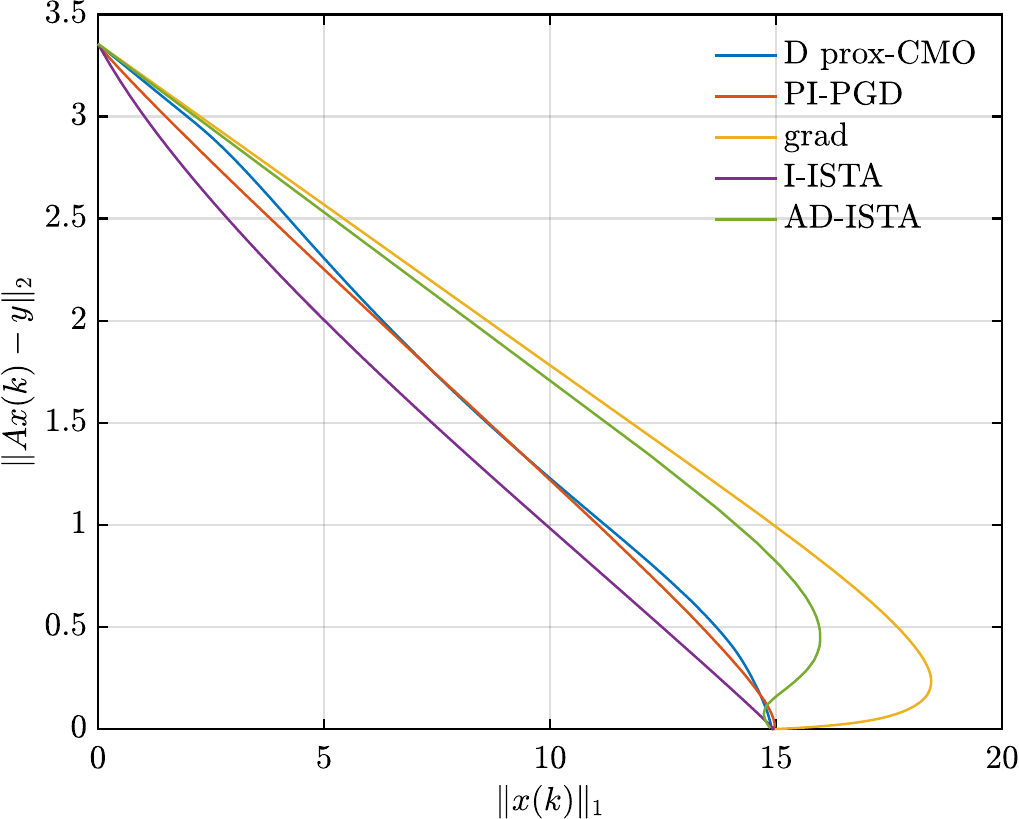}
    \caption{Residual $\| Ax - y\|_2$ versus $\| x \|_1$ averaged over 100 runs. Comparison between the proposed dynamic Prox-CMO, I-ISTA \cite{SIC_IISTA_25}, PI-PGD \cite{cent25}, AD-ISTA \cite{fox23}, and the gradient descent method.}
    \label{fig: lasso cmo res vs l1}
\end{figure}

\begin{figure}[ht]
    \centering
    \includegraphics[width=\linewidth]{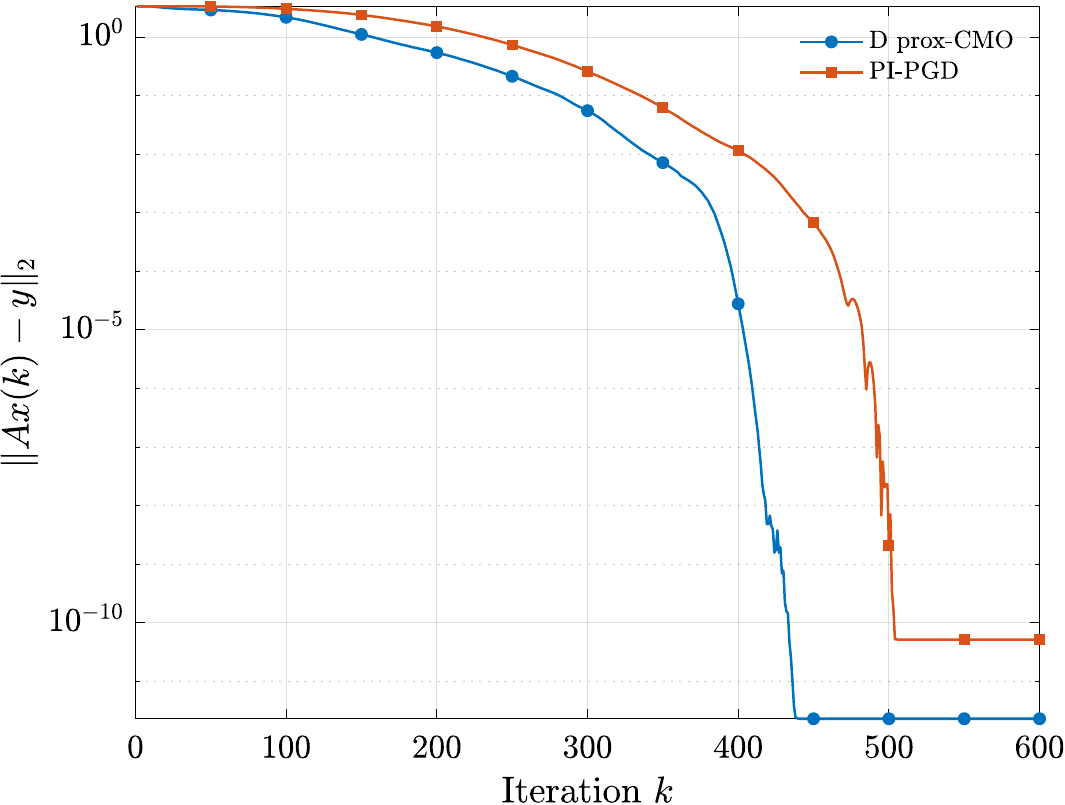}
    \caption{Evolution of the residual of the CT algorithms $\|Ax(k)-y\|_2$ averaged over 100 runs.}
    \label{fig: cmo1}
\end{figure}

\begin{figure}[ht]
    \centering
    \includegraphics[width=\linewidth]{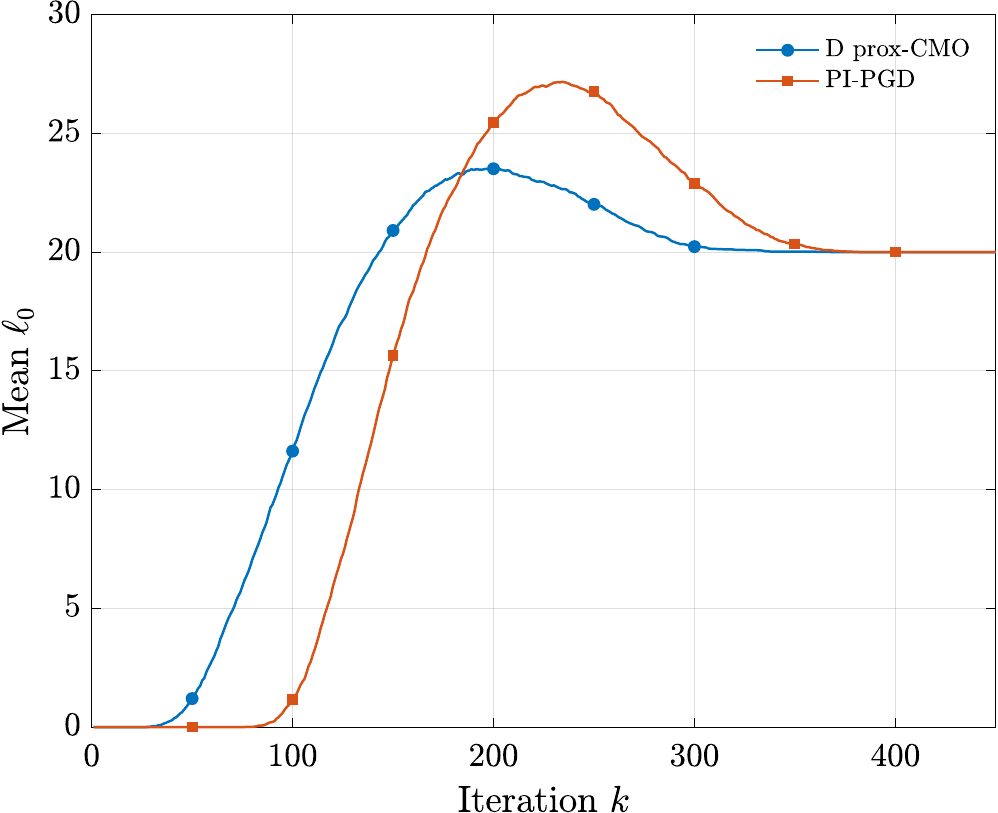}
    \caption{Evolution of the sparsity level $\|x(k)\|_0$ averaged over 100 runs.}
    \label{fig: cmo2}
\end{figure}

\begin{figure}[ht]
    \centering
    \includegraphics[width=\linewidth]{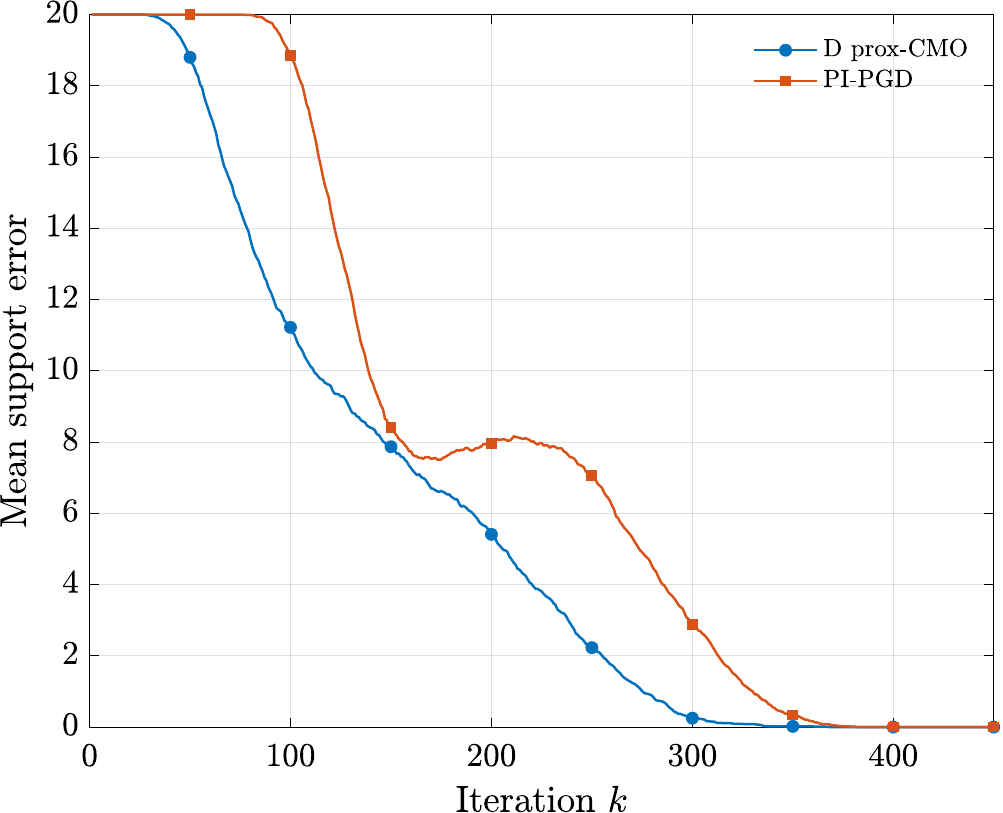}
    \caption{Evolution of the support error $\sum_{i=1}^n |\iota(x_i(k))-\iota(\xtrue_i)|$ averaged over 100 runs.}
    \label{fig: cmo3}
\end{figure}


\begin{table}[t]
\centering
\begin{tabular}{lccc}
\toprule
\textbf{Alg.} &  \textbf{Residual} & \textbf{Time [s]} \\
\midrule
D Prox-CMO &$1.4\cdot10^{-10}$ & 0.177 \\
PI-PGD        &$4.6\cdot10^{-9}$   & 0.187 \\
Grad. desc.   &$5.1\cdot10^{-10}$   & 0.707 \\
I-ISTA        &$3.8 \cdot10^{-12}$   & 0.223 \\
\bottomrule
\end{tabular}
\caption{Average residual, computational time, and number of iterations over 100 random runs.}
\label{table:confront_algs}
\end{table}

All the results are averaged over 100 random runs.
Figure \ref{fig: lasso cmo res vs l1} illustrates the trajectories of the residuals $\| Ax-y \|_2$ versus the $\ell_1$ norms$ \| x \|_1$.
The proposed Dynamic Prox-CMO method, together with I-ISTA and PI-PGD, exhibits approximately linear trajectories in the $\| Ax-y \|_2$ - $\| x \|_1$ plane, indicating an effective tradeoff between residual minimization and sparsity promotion. 
For comparison, we also show the trajectories of gradient descent applied to the objective $f(x)$ and of AD-ISTA \cite{fox23}, a fast version of proximal gradient descent for Lasso. 
The gradient method presents a pronounced $\ell_1$ overshoot, which is only partially reduced by AD-ISTA; see \cite{fox23} for further discussion. This overshoot  can be critical in applications such as secure state estimation in cyber-physical systems, where a transient large number of false positives may trigger unnecessary mitigation actions and degrade system performance; see, e.g., \cite{SIC_IISTA_25}. 

Table~\ref{table:confront_algs} summarizes the average computational time required to reach convergence, along with the final residual values. In the proposed setting, all the algorithms reach similar residuals, effectively eliminating bias. However, dynamic Prox-CMO converges faster on average than all the other presented algorithms.

Finally, Figures~\ref{fig: cmo1}, \ref{fig: cmo2}, \ref{fig: cmo3} present a more detailed comparison of dynamic Prox-CMO and PI-PGD.
Fig.~\ref{fig: cmo1} shows the evolution of the residual $\|Ax(k)-y\|_2$ over the number of iterations. Fig.~\ref{fig: cmo2} depicts instead the evolution of the sparsity level $\|x(k)\|_0$. While both approaches reach the correct sparsity level, we observe that PI-PDG requires more iterations to settle, and is more prone to overshoot, i.e., to generate false positives in the estimation of the support.
Finally, in Fig.~\ref{fig: cmo3} we plot the mean support error, that is defined as $\sum_{i=1}^{n} |\iota(x_i(k)) - \iota(\tilde{x}_i)|$, where \( \iota \) denotes the indicator function \( \iota(z) = \|z\|_0 \) for \( z \in \mathbb{R} \). 
Both approaches reach the correct support, but again we observe that dynamic Prox-CMO has a faster convergence.

\subsection{Shidoku puzzle}
The second numerical example is a problem with a nonsmooth cost function and non-convex polynomial constraints.

Shidoku is a 4x4 version of the 9x9 Sudoku puzzle. Given an initial scheme as the one reported in Fig.~\ref{fig: shidoku}, the aim is to fill the empty cells with integers $ x_{i,j} \in \{1, 2, 3, 4\} $ such that each row, each column, and each 2x2 corner block contains them without repetitions.

\begin{figure}[h!]
\centering

\begin{minipage}[t]{0.48\linewidth}
\centering
\begin{tikzpicture}

\draw[line width=0.8pt] (0,0) rectangle (4,4);
\foreach \x in {1,2,3}
    \draw[line width=0.4pt] (\x,0) -- (\x,4);
\foreach \y in {1,2,3}
    \draw[line width=0.4pt] (0,\y) -- (4,\y);

\node at (1.5,3.5) {\textbf{1}};
\node at (3.5,3.5) {\textbf{4}};
\node at (0.5,1.5) {\textbf{2}};
\node at (2.5,1.5) {\textbf{3}};

\end{tikzpicture}

\small (a) Initial scheme
\end{minipage}\hfill
\begin{minipage}[t]{0.48\linewidth}
\centering
\begin{tikzpicture}

\draw[line width=0.8pt] (0,0) rectangle (4,4);
\foreach \x in {1,2,3}
    \draw[line width=0.4pt] (\x,0) -- (\x,4);
\foreach \y in {1,2,3}
    \draw[line width=0.4pt] (0,\y) -- (4,\y);

\node at (0.5,3.5) {3};
\node at (1.5,3.5) {\textbf{1}};
\node at (2.5,3.5) {2};
\node at (3.5,3.5) {\textbf{4}};

\node at (0.5,2.5) {4};
\node at (1.5,2.5) {2};
\node at (2.5,2.5) {1};
\node at (3.5,2.5) {3};

\node at (0.5,1.5) {\textbf{2}};
\node at (1.5,1.5) {4};
\node at (2.5,1.5) {\textbf{3}};
\node at (3.5,1.5) {1};

\node at (0.5,0.5) {1};
\node at (1.5,0.5) {3};
\node at (2.5,0.5) {4};
\node at (3.5,0.5) {2};

\end{tikzpicture}

\small (b) Solved scheme
\end{minipage}

\caption{Shidoku puzzle and its solution. Bold entries indicate the original givens.}
\label{fig: shidoku}
\end{figure}

We formulate the game as a constrained, nonsmooth optimization problem with variables $ x = \{ x_{i,j} \} $ where $ (i,j), \ i,j = 1,\dots,4 $ are the indices of each cell. 
We avoid repetitions in rows, columns and blocks by imposing the product of the elements equal to 24 and the sum equal to 10. We list below the elements composing the non-convex constraints $h(x)$.

Columns: for $j = 1, \ldots, 4$
\[
\sum_{i=0}^{4} x_{ij} = 10, \qquad \prod_{i=0}^{4} x_{ij} = 24;
\]
rows: for $i = 1, \ldots, 4$
\[
\sum_{j=0}^{4} x_{ij} = 10, \qquad \prod_{j=0}^{4} x_{ij} = 24;
\]
blocks: for $k = 1, \ldots, 4$
\[
\sum_{(i,j) \in B_k} x_{ij} = 10, \qquad \prod_{(i,j) \in B_k} x_{ij} = 24.
\]
Finally, the initial conditions are the ones in Fig.~\ref{fig: shidoku}
\[
x_{1,2} = 1, \quad x_{1,4} = 4, \quad x_{3,1} = 2, \quad x_{3,4} = 3.
\]
Moreover, condition $ x_{i,j} \in \mathcal{C} = \{ x \in \mathbb{N}: 1 \leq x \leq 4\} \quad \forall i,j$ can be guaranteed exploiting the indicator function $\iota_{\mathcal{C}}$. Thus, the optimization problem we address is:

\begin{equation*}
    \begin{aligned}
        \min_{x \in \R^{n}} \quad \iota_{\mathcal{C}}(x) \\
        \textrm{s.t.} \quad & h(x)  = 0
    \end{aligned}
\end{equation*}

The proximal operator of $ \iota_{\mathcal{C}}(x)$ is the projection on the set $\mathcal{C}$, computed as:
\begin{equation}
\Pi(x) = 
\begin{cases}
1 & x \leq 1.5, \\
2 & 1.5 < x \leq 2.5, \\
3 & 2.5 < x \leq 3.5, \\
4 & x > 3.5.
\end{cases}
\end{equation}

In \cite{SIC_TAC_25}, authors recast this problem into a set of polynomial equations in the variables $x_{i,j}$ and solved it using PI-CMO.
The primary difference between the approach in \cite{SIC_TAC_25} and the one proposed here lies in the way the constraints $x_{i,j} \in \mathcal{C} = \{ x \in \mathbb{N}: 1 \leq x \leq 4\}$ are enforced: while in \cite{SIC_TAC_25} they are recast in additional terms for $h(x)$: $\prod_{h=1}^{4} (x_{ij} - h) = 0$, in this work we include them in the cost function. This conceptual modification, enabled by the nonsmooth formulation, leads to a faster and more computationally efficient solution.

For the simulations, we set $ k_i = 1, \ k_p = 0.1 $ for PI-CMO, $ \mu = 1, k_p = 0.1, k_i = 1, k_1 = -0.1, k_2 = -1, k_3 = 0.9 $ for dynamic Prox-CMO and $ \mu= 4,  \ k_{i} = 1, \ k_{p} = 2 $ for static Prox-CMO.
We generate random initial conditions for all the unknows $x_{i,j}$ according to $ x_{i,j}(0) \sim |\mathcal{N}(0,1)| $. Zero initial conditions are instead employed for all sets of Lagrange multipliers. 
We integrate the ordinary differential equations that describe the closed-loop dynamics using MATLAB \texttt{ode15s} solver in the time interval $[0,100]$. All the algorithms converge to the correct solution of the scheme.

Table \ref{tab:shidoku_results} collects the number of iterations and computational time averaged over 50 runs of the three considered algorithms.
We observe that both Prox-CMO algorithms require fewer iterations than PI-CMO, and in particular, the static Prox-CMO cuts down the computational times significantly when compared to the dynamic version.

In this application PI-PGD fails to converge to the correct solution. This behavior is likely due to the algorithm's formulation: as discussed in Sec. \ref{sec: comp with pipgd}, the constraints appear within the $\prox$ operator, which in this case is the projection onto the set $\mathcal{C}$. In this example, this formulation appears to induce numerical instability, which is the most plausible explanation for the observed lack of convergence.
\begin{table}[t]
\centering
\begin{tabular}{lccc}
\toprule
\textbf{Alg.} & \textbf{Var.} & \textbf{Iter.} & \textbf{Time [s]} \\
\midrule
PI-CMO        & 56 & 2697.9 & 0.4854 \\
D Prox-CMO    & 56 & 1563.1 & 0.3090 \\
S Prox-CMO    & 44 & 1313.1 & 0.1104 \\
\bottomrule
\end{tabular}
\caption{Performance comparison of PI-CMO, dynamic Prox-CMO, and static Prox-CMO in the Shidoku problem.}
\label{tab:shidoku_results}
\end{table}

\subsection{Set-membership system identification}\label{sec: set memb}

The last numerical example considers a set-membership system identification problem.

We aim at identifying a discrete-time system described by the transfer function

\begin{equation}
    H(z) = \frac{1}{(z-0.56)(z-0.78)}.
\end{equation}

It is a second-order, stable LTI system, suited to model damped and decaying dynamics.
The true system response is generated by exciting $H(z)$ with a uniformly distributed random input signal $u$. 
The measured output is affected by a random additive noise sequence $\tilde{y} = y_{\text{true}} + \eta$ satisfying $ \|\eta\|_{\infty} \leq \gamma $ and $ \|\eta\|_{2} \leq \varepsilon $.
The regression model is built from a basis of $d$ Laguerre transfer functions \cite{masnadi91} that offer a compact, orthonormal basis for stable LTI systems with decaying dynamics. 
For a Laguerre parameter $a \in (0,1)$, the basis functions are generate as

\begin{align*}
   & B_1(z) = \frac{\sqrt{1-a^2}}{1 - az^{-1}}, \\[3pt]
   & B_i(z) = \left( \frac{z^{-1} - a}{1 - az^{-1}} \right)  B_{i-1}(z), \quad i=2,\dots,d.
\end{align*}
    
We select $a = 0.75, d=5$ based on a grid search.

The regression matrix $\Phi \in \R^{N \times d}$ is obtained by simulating the response of each basis function $B_i(z)$ in the presence of the input sequence $u$. Thus, the predicted output $\hat{y}$ is defined as $\hat{y} = \theta \Phi$.
Our aim is to find the feasible parameters set $[\theta_{\text{lower}},\theta_{\text{upper}}]$ that best approximates the true system dynamics in the presence of noise.
The noise $\eta$ must belong to the  set 
\begin{equation}
    \mathcal{C} = \{ \eta \in \mathbb{R}^N : \|\eta\|_{\infty} \leq \gamma, \; \|\eta\|_{2} \leq \varepsilon \}.
\end{equation}

Based on these assumptions, the optimization problem can be formulated as:
\begin{equation}
    \begin{aligned}
            \min_{\theta} &\quad \pm \theta_i + \iota_{\mathcal{C}}(\eta) \\
            &\text{s.t.} \quad \eta = \tilde{y} - \Phi \theta
    \end{aligned}    
\end{equation}
where $ \iota_{\mathcal{C}}(\eta) $ denotes the indicator function of set $ \mathcal{C} $. 

The proximal operator associated with $\iota_{\mathcal{C}}(\eta)$ is the projection onto the set $\mathcal{C}.$ It is obtained by means of Dykstra's projection algorithm \cite{dykstra86}, which allows computation of a point in the intersection of two convex sets.

For the numerical simulations, we set $ N = 50 $, $ \gamma = 1.5 \|\eta\|_{\infty} $, and $ \varepsilon = 1.7 \|\eta\|_{2} $. 
The integrations of both dynamic and static prox–CMO algorithms are performed using MATLAB’s \texttt{ode15s} solver.
Given the final bounds $\theta_{\text{lower}}$ and $\theta_{\text{upper}}$, a nominal estimate for vector $\hat{\theta}$ is obtained as the average $\hat{\theta} = \tfrac{1}{2} \big(\theta_{\text{lower}} + \theta_{\text{upper}}\big)$.

Given a test dataset of $N_{\text{test}} = 1000$ randomly generated points, we compute the predicted output $\hat{y}(k)$ using $\hat{\theta}$ and compare it to the true system response $y(k)$. 
The performance metric used is the fit percentage:
\begin{equation}
    \text{FIT} = 100 \left( 1 - \sqrt{ \frac{\lVert y - \hat{y}\rVert}{\lVert y- \overline{y} \rVert}} \right)
\%.
\end{equation}
where we denote as $\overline{y}$ the average of the test output.

\begin{table}[t]
\centering
\begin{tabular}{lccc}
\toprule
\textbf{Param.} & $\boldsymbol{\theta_{\mathrm{low}}}$ & $\boldsymbol{\theta_{\mathrm{up}}}$ & $\boldsymbol{\hat{\theta}}$ \\
\midrule
$\theta_1$ & 1.8530  & 2.3666  & 2.1098 \\
$\theta_2$ & 1.9673  & 2.5599  & 2.2636 \\
$\theta_3$ & -0.9249 & -0.2922 & -0.6085 \\
$\theta_4$ & -0.0981 & 0.5156  & 0.2088 \\
$\theta_5$ & -0.2446 & 0.2340  & -0.0053 \\
\bottomrule
\end{tabular}
\caption{Estimated parameter values with lower and upper bounds.}
\label{tab:parameters}
\end{table}

\begin{table}[t]
\centering
\begin{tabular}{lc}
\toprule
\textbf{Method} & \textbf{Time [s]} \\
\midrule
CVX              & 1.99 \\
Dyn. Prox-CMO    & 0.35 \\
Stat. Prox-CMO   & 0.33 \\
PI-PGD           & 0.40 \\
\bottomrule
\end{tabular}
\caption{Average computational time over 100 runs.}
\label{tab:comp_time}
\end{table}

In the dynamic Prox–CMO the parameters are set to $ \mu = 15, \; k_p = 3, \; k_i = 0.1, \; k_1 = -2, \; k_2 =-1, \; k_3 = -1$, while for static Prox-CMO we set $ \mu = 0.05, \; k_p =0.7, \; k_i = 0.1$. We compare our algorithms to PI-PGD, where we set $ \gamma = 1, k_p = 1, k_i = 1.5$.
The resulting bounds and nominal estimates obtained using the \texttt{cvx} solver and the other algorithms are reported in Table~\ref{tab:parameters}. All methods achieve a FIT value of 96.73\%. The average computational time over 100 runs is reported in Table~\ref{tab:comp_time}. The results indicate that all CMO-based algorithms converge faster than \texttt{cvx}. In particular, both Prox-CMO variants exhibit faster convergence compared to PI-PGD.

\section{Conclusions}
\label{sec:conc}
This paper presents two algorithms based on control theory and designed to address equality-constrained nonsmooth composite optimization problems. After introducing a continuously differentiable proximal augmented Lagrangian, we employ the controlled multipliers optimization approach to define a dynamical system associated with the problem, using the Lagrange multipliers as control inputs to drive the system toward an equilibrium point. 
In particular, we focus on the multipliers corresponding to the non-differentiable term and we develop two feedback controllers - a static one and a dynamic one - to produce effective control inputs.

The static approach gives rise to an algorithm that is an extension of the proximal gradient method to constrained optimization. In contrast, the dynamic approach extends the nonsmooth primal–dual gradient descent. 

For both methods, we establish global exponential convergence in the presence of a strongly convex cost function and linear constraints. 

Experimental results support the theoretical results and demonstrate the effectiveness of the proposed framework, also in nonconvex problems.

Future work will address the extension of the proposed approach to distributed optimization over networks.
\bibliographystyle{IEEEtran}  
\bibliography{bibliography}  

\end{document}